\documentclass[oneside,a4paper,reqno,psamsfonts]{amsart}
\usepackage[latin1]{inputenc} 
\usepackage{amscd}
\usepackage{amsopn}
\usepackage{amstext}
\usepackage{amsxtra}
\usepackage{amssymb}
\usepackage{stmaryrd}
\usepackage{upref}
\usepackage{url}
\usepackage{textcase} 
\usepackage{bm} 
\usepackage{graphicx}
\usepackage{setspace}
\usepackage{caption}
\usepackage{placeins}
\numberwithin{algorithm}{section}

\theoremstyle{plain}
\newtheorem{thm}{Theorem}
\newtheorem{lem}[thm]{Lemma}
\newtheorem{cl}[thm]{Corollary}

\theoremstyle{definition}

\theoremstyle{remark}

\newtheorem{rem}[thm]{Remark}

\numberwithin{thm}{section}

\providecommand{\abs}[1]{\left\lvert #1 \right\rvert}

\providecommand{\ceil}[1]{\left\lceil #1 \right\rceil}

\providecommand{\set}[1]{\left\lbrace #1 \right\rbrace}
\providecommand{\gen}[1]{\left\langle #1 \right\rangle}

\newcommand{\field}[1]{\mathbb{#1}}

\newcommand{\N}{\field{N}}
\newcommand{\Z}{\field{Z}}

\newcommand{\F}{\field{F}}
\newcommand{\C}{\field{C}}

\newcommand{\MAGMA}{\textsc{Magma}}

\newcommand{\OV}{\mathcal{O}}

\DeclareMathOperator{\GL}{GL}

\DeclareMathOperator{\Sz}{Sz}

\DeclareMathOperator{\Tr}{Tr}
\DeclareMathOperator{\SLP}{\mathrm{SLP}}

\DeclareMathOperator{\Norm}{N}
\DeclareMathOperator{\Cent}{C}
\DeclareMathOperator{\Zent}{Z}

\DeclareMathOperator{\cln}{{:}}

\DeclareMathOperator{\Dih}{D}

\def\cn{\mathord{{\!\:{:}\:\!}}}

\newcommand{\OR}[1]{\operatorname{O} \bigl( #1 \bigr)}
\newcommand{\oR}[1]{\operatorname{o} \bigl( #1 \bigr)}
\newcommand{\TR}[1]{\operatorname{\Theta} \bigl( #1 \bigr)}

\title{Black box recognition of the Suzuki groups}
\author{John N. Bray}
\author{Henrik B\"a\"arnhielm}
\email{jbaa004@aucklanduni.ac.nz}
\address{Department of Mathematics \\ University of Auckland \\ Private Bag 92019 \\ Auckland \\ New Zealand}

\begin{document}

\begin{abstract}
  We present a black box algorithm that constructs standard generators
  for the Suzuki groups $\Sz(q)$, where $q = 2^{2m + 1}$ for some $m >
  0$. The algorithm is one-sided Monte Carlo, with no false positives,
  with time complexity $\OR{q (\log q)^6}$ group operations. We also
  present a black box algorithm, with time complexity $\TR{q}$ group
  operations, that performs constructive membership testing in
  $\Sz(q)$, and writes an element as a straight line program in the
  standard generators. Finally, we give a presentation for $\Sz(q)$
  that is efficient to verify. The algorithms have been implemented in
  the computer algebra system $\MAGMA$.
\end{abstract}

\maketitle

\section{Introduction}
\label{section:intro}

The family of finite simple groups known as the Suzuki groups was
introduced in \cite{MR0120283, suzuki62, MR0162840}, as matrix groups
of degree $4$.  The groups in this family are usually denoted
$\Sz(q)$, where $\F_q$ is the defining field, and $q = 2^{2m + 1}$ for
some $m > 0$. The following are our designated standard generators for
$\Sz(q)$.
\begin{align}
\label{x_std_gen} \hat{x} &= \begin{bmatrix}
1 & 0 & 0 & 0 \\
1 & 1 & 0 & 0 \\
1 & 1 & 1 & 0 \\
1 & 0 & 1 & 1
\end{bmatrix}, \\
\label{y_std_gen} \hat{y} &= \begin{bmatrix}
\omega^{t/2 + 1} & 0 & 0 & 0 \\
0 & \omega^{t/2} & 0 & 0 \\
0 & 0 & \omega^{-t/2} & 0 \\
0 & 0 & 0 & \omega^{-t/2 - 1} 
\end{bmatrix}, \\
\label{z_std_gen} \hat{z} &= \begin{bmatrix}
0 & 0 & 0 & 1 \\
0 & 0 & 1 & 0 \\
0 & 1 & 0 & 0 \\
1 & 0 & 0 & 0
\end{bmatrix},
\end{align}
where $t = 2^{m + 1} = \sqrt{2q}$ and $\omega$ is a primitive element
of $\F_q$. These standard generators are defined up to conjugation by
an element of the automorphism group of $\Sz(q)$ and up to
specification of $\omega$. In this paper we
present a black box algorithm that constructs the above standard
generators in a group isomorphic to $\Sz(q)$, and an algorithm that
writes an arbitrary element as a word in the standard generators.

Constructive recognition and membership algorithms for the natural and
defining characteristic absolutely irreducible representations of
$\Sz(q)$ are presented in \cite{baarnhielm_phd, baarnhielm05,
  MR3715222}. Here, for the first time, we deal with black box
representations of $\Sz(q)$. One motivation is to contribute to the
Matrix Group Recognition Project \cite{MR3283836, crlg01}.

The algorithms presented here are not black box versions of the
earlier algorithms, but are novel algorithms based on new ideas. In
practice these new algorithms are useful primarily for input
representations in non-defining characteristic, where $q$ is
small. They are the only such algorithms with proven time complexity
bounds for non-natural representations. In the natural representation,
the algorithms of \cite{MR3715222} are more practical.

In \cite[p. 17]{seress03}, the concept of a \emph{black box group} is
defined. In short, it is a group where the elements are encoded as
strings of uniform length over a finite alphabet, equipped with an
\emph{oracle} (the \lq\lq black box'') that can construct a string
representing the product of two given elements, construct a string
representing the inverse of a given element, and test if a given
element is the identity. 

Recall from \cite{clr90} and \cite{seress03} the standard time complexity notation $\OR{\cdot}$, $\oR{\cdot}$ and $\TR{\cdot}$.
The time complexity measures are given in bit operations. We will use
the following notation for a black box group $G$:
\begin{itemize}
\item[$\mu$] The time complexity of the group operation in $G$, \emph{i.e.} the number of bit operations required by the black box.
\item[$\eta$] The time complexity required to compute the order of an element of $G$.
\item[$\xi$] The time complexity required to construct a uniformly distributed random element of $G$.
\item[$\zeta$] The time complexity of a field operation, \emph{i.e.} a multiplication in $\F_q$.
\item[$\chi_D$] The time complexity required to solve an instance of the discrete logarithm problem in $\F_q$.
\item[$\chi_F$] The time complexity required to factorise $\abs{\Sz(q)}$.
\end{itemize}

This notation is used to make the algorithms and time complexity more
clear. Some is redundant:
\begin{itemize}
\item Clearly, $\chi_D \in \OR{q \zeta}$, and since $q$ is a power of $2$, by \cite{cop84}, 
\[ \chi_D \in \OR{\exp(c (\log_2 q)^{1/3} (\log\log_2 q)^{2/3})} \]
where $c > 0$ is a small constant. 
\item Clearly $\chi_F \in \OR{q}$ since $\abs{\Sz(q)} = q^2(q^2 + 1)(q-1)$. Moreover, by \cite{MR1416721}, we can use the \emph{special number field sieve} to obtain 
\[ \chi_F \in \OR{\exp((32\log_2\abs{\Sz(q)} / 9)^{1/3} (\log\log_2\abs{\Sz(q)})^{2/3})}.\]
\item Since $\log \abs{G} \in \OR{\log q}$, it follows from \cite{babai91} that $\xi \in \OR{(\log q)^5 \mu}$. 
\item Every element of $\Sz(q)$ has order $\OR{q}$, so that $\eta \in \OR{q \mu}$. Since element orders divide some number in $\set{4, q - 1, q \pm t + 1}$, the algorithm in \cite{crlg95} shows that $\eta \in \OR{\chi_F + \mu \log q \log\log q}$. 
\item Since elements of $\F_q$ are represented as bit-strings, by \cite[Theorem $8.23$]{VonzurGathen03}, $\zeta \in \OR{\log q \log\log q \log\log\log q}$.
\end{itemize}



The algorithm presented here will find standard generators expressed
as \emph{straight line programs} (abbreviated to SLPs) in the given
generating set. An SLP is a data structure for a word, a list of
cells, with each cell representing an operation such as a group
multiplication, see \cite[Section
  $5$]{Baarnhielm:2012:PRP:2069778.2070144}. Using it one can ensure
that, during evaluation, subwords occurring multiple times are not
computed more often than during construction. An important issue is
the length of the SLPs that are computed. We assume that SLPs of
random elements have length $\OR{n}$ where $n$ is the number of random
elements that have been selected so far during the execution of the
algorithm. Multiplication and inversion increases the length of an SLP
by $1$. Raising an element to a power $k$ increases the length of an
SLP by $\log(k)$ since powering is done using the repeated squaring
technique.

Our main result is the following.

\begin{thm} \label{thm_sz_std_gens}
Assume that $\F_q$, together with a primitive element, has been constructed, where $q = 2^{2m + 1}$ for some $m > 0$, and let $\hat{x}$, $\hat{y}$ and $\hat{z}$ be as in Equations \eqref{x_std_gen}, \eqref{y_std_gen} and \eqref{z_std_gen}. There exists a one-sided Monte Carlo algorithm with no false positives that, given a black box group $G = \gen{X}$ such that $G$ is simple, and $\varepsilon > 0$, determines whether $G \cong \Sz(q)$, and if so constructs $x, y, z \in G$ as straight line programs in $X$, such that the mapping 
\begin{align*}
x &\mapsto \hat{x} & y &\mapsto \hat{y} & z &\mapsto \hat{z}
\end{align*}
induces an isomorphism. The probability of failing to recognise that $G \cong \Sz(q)$ is less than $\varepsilon$. The algorithm has time complexity
$\OR{q\log(q)(\xi + \mu) + (\log\log q) \log(q) \eta  + q(\log q)^2 \zeta + \chi_F}$, or equivalently $\OR{q
  (\log q)^6 \mu}$. The length of the returned $\SLP$s is $\OR{q}$. 
\end{thm}

Note that the algorithm in Theorem \ref{thm_sz_std_gens} is intended
to be used within \cite{MR3283836}. In this application, we learn the
value of $q$ in polynomial time using \cite{general_recognition,
  blackbox_char} and other algorithms are used for the non-simple
group case, which justifies the assumption in Theorem
\ref{thm_sz_std_gens} that the input group is simple.

\begin{thm} \label{thm_sz_membership}
Assume that $\F_q$, together with a primitive element, has been constructed, where $q = 2^{2m + 1}$ for some $m > 0$. There exists a deterministic algorithm that, given 
\begin{itemize}
\item a simple black box group $G = \gen{X}$ such that $G \cong \Sz(q)$, 
\item standard generators $Y = \set{x, y, z} \subset G \cong \Sz(q)$ as found by the algorithm in Theorem \ref{thm_sz_std_gens},
\item $g \in U \geqslant G$ for some overgroup $U$
\end{itemize}
decides whether or not $g \in G$, and if so, expresses $g$ as an
$\SLP$ in $Y$. The algorithm has time complexity $\TR{q \mu}$. The
length of the returned $\SLP$ is $\OR{\log q}$.
\end{thm}


The algorithms in Theorems \ref{thm_sz_std_gens} and \ref{thm_sz_membership} are implemented in the computer
algebra system $\MAGMA$ \cite{magma}.

An important feature of the constructive recognition algorithm in Theorem \ref{thm_sz_std_gens}
is that its time complexity is $\oR{q^2}$. It is straightforward to
obtain an algorithm that constructs standard generators in time $\OR{q^3}$, as follows. One constructs a point stabiliser by constructing two random elements of order
$q - 1$ that have a common fixed point, in expected time $\OR{q^2}$.
Then elements corresponding to $\hat{x}$ and $\hat{y}$ are easily
constructed within this point stabiliser. The element corresponding to
$\hat{z}$ is then constructed by taking random conjugates of $\hat{x}^2$,
in expected time $\OR{q^3}$.

The first obstacle to obtain running time $\oR{q^2}$ is therefore to
quickly construct a point stabiliser. This is done in Theorem
\ref{thm_bray_stab_trick}, which involves the algorithm of
\cite{bray00}, and the \lq\lq dihedral trick'' \cite[Theorem 1.9]{baarnhielm_phd} for conjugating involutions to each other.

To reduce the running time of the rest of the constructive recognitoin algorithm, we use crucial
relations from a presentation of $\Sz(q)$, and its properties in the natural represenation. These are described in Section \ref{bray_sz_presentation}. An outline of the constructive recognition algorithm in Theorem \ref{thm_sz_std_gens} is as follows:
\begin{itemize}
\item Construct an element $f$ of order $4$ by random search. This is step \ref{main_alg_elt4_step} in Section \ref{section:main_alg}. From $f$ we construct $z$, step \ref{main_alg_z_step} in Section \ref{section:main_alg}.
\item Construct a point stabiliser from $f$, and hence construct an element $h$ of order $q - 1$. This is step \ref{main_alg_stab_step} in Section \ref{section:main_alg}.
\item Replace $h$ so that its diagonal part corresponds to the correct primitive element of $\F_q$. This is steps \ref{main_alg_f_standard_step} to \ref{main_alg_h_standard_step} in Section \ref{section:main_alg}.
\item Obtain $x$ from $f$ by conjugation with $h$. This is step \ref{main_alg_x_step} in Section \ref{section:main_alg}.
\item Construct $y$ from $x$ and $z$ using a particular relation. This steps \ref{main_alg_omega_step} to \ref{main_alg_y_step} in Section \ref{section:main_alg}.
\end{itemize}

The constructive membership algorithm in Theorem \ref{thm_sz_membership} is
conceptually simpler. It mainly relies on the fact that $\Sz(q)$ is a
disjoint union of two double cosets of its unique maximal parabolic subgroup. This is
used to reduce the membership testing to the maximal parabolic.
Theorem \ref{thm_parabolic_slp} solves the membership problem, which is
essentially Gaussian elimination in the black box context.

\section{Preliminaries}
\label{section:preliminaries}

Recall that the \emph{trace} (over the prime field) of $a
\in \F_q$, where $q = p^e$ for some prime $p$, is $\Tr(a) =
\sum_{i = 0}^{e - 1} a^{p^i} \in \F_p$. Trivially, $\Tr(a)$ can be computed in time $\OR{e p \zeta}$.

Henceforth assume that $q = 2^{2m + 1}$ for some $m > 0$, $t = 2^{m + 1} = \sqrt{2q}$, let $\omega$ denote a fixed primitive element of $\F_q$ and let $\hat{x}$, $\hat{y}$ and $\hat{z}$ be as in Equations \eqref{x_std_gen}, \eqref{y_std_gen} and \eqref{z_std_gen}. We shall
also denote the Euler totient function by $\phi$. 

Elements of $\F_q$ are represented as polynomials over $\F_2$, \emph{i.e.} as bit-strings. Therefore we can express every $a \in \F_q$ as $a = \sum_{i = 0}^{2m} a_i
\omega^i$, where each $a_i \in \F_2$. Then 
\begin{equation} \label{eqn_trace_comp}
\Tr(a) = \sum_{j = 0}^{2m} a^{2^j} = \sum_{j = 0}^{2m} \sum_{i = 0}^{2m} a_i (\omega^i)^{2^j} = \sum_{i = 0}^{2m} a_i \Tr(\omega^i)
\end{equation}
and hence if $\Tr(\omega^i)$ is pre-computed for $i = 0, \dotsc, 2m$, then $\Tr(a)$ can be computed in time $\OR{\log q}$.

\begin{lem} \label{lem_primitive_elt}
There exists $k \in \set{1, \dotsc, q - 1}$ such that 
\begin{itemize}
\item $\gcd(k, q - 1) = 1$,
\item $\Tr(\omega^{-k}) = 1$.
\end{itemize}
\end{lem}
\begin{proof}
This is proved in \cite{MR676862}.
\end{proof}

\begin{lem} \label{lem_quadratics}
Let $a \in \F_q^{\times}$.
\begin{itemize}
\item 
The quadratic equation (an element of $\F_q[b]$)
\begin{equation}  \label{eq:lem_quad1}
b^2 + a^{t + 1} b + a^2 + a^{2t} = 0
\end{equation}
has solutions $b_1 = a^{t+1} \sum_{i=1}^{m+1} a^{-2^i}$ and $b_2 = b_1 + a^{t + 1}$.
\item Let $c \in \F_q^{\times}$. The quadratic equation
\begin{equation} \label{eq:lem_quad2}
b^2 + a^{t + 1} b + c^{-1} + a^t c^{-t/2} = 0
\end{equation}
has solutions $b_1 = a^{t + 1} \sum_{i = 0}^m \lambda^{2^i}$ and $b_2
= b_1 + a^{t + 1}$, where $\lambda = c^{-t / 2} a^{-t-2} $.
\end{itemize}
\end{lem}
\begin{proof}
  Observe first that if $b^2 + a^{t + 1} b + y = 0$, then $(b + a^{t +
    1})^2 + a^{t + 1}(b + a^{t + 1}) + y = b^2 + a^{t + 1} b + y = 0$,
  so if $b$ is any solution then $b + a^{t + 1}$ is also a solution.

  Now observe that if $x_1 = \sum_{i = 0}^{r - 1} \lambda^{2^i}$ then
  $x_1^2 + x_1 = \lambda + \lambda^{2^r}$, so that $x_1$ is a solution
  to the quadratic equation
\begin{equation} \label{eq:lem_quad3}
  x^2 + x + \lambda + \lambda^{2^r} = 0.
\end{equation}
\begin{itemize}
\item The substitution $b = a^{t + 1} x$ transforms Equation \eqref{eq:lem_quad1} into $x^2 + x + a^{-2} + a^{-2t} = 0$. This is Equation \eqref{eq:lem_quad3} with $\lambda = a^{-2}$ so the solutions are given in the previous paragraph.

\item Similarly, the substitution $b = a^{t + 1} x$ transforms Equation \eqref{eq:lem_quad2} into $x^2 + x + c^{-1} a^{-2-2t} + c^{-t/2} a^{-t-2} = 0$. Since $c^{-1} a^{-2-2t} = (c^{-t/2} a^{-t-2})^t$, this is Equation \eqref{eq:lem_quad3} with $\lambda = c^{-t/2} a^{-t-2}$ so the solutions are given by the previous paragraph.
\end{itemize}
This completes the proof.
\end{proof}

Assume $F$ is a field of characteristic $2$ equipped with an automorphism $\sigma$ that satisfies $\lambda^{\sigma^2}=\lambda^2$ for all $\lambda\in F$.
Henceforth we shall use the following definitions, where $a,b \in F$ and $k \in F^{\times}$.
\begin{align*}
T(a, b) &= \begin{bmatrix}
1 & 0 & 0 & 0 \\
a & 1 & 0 & 0 \\
a^{\sigma + 1} + b & a^\sigma & 1 & 0 \\
a^{\sigma + 2} + ab + b^\sigma & b  & a & 1 
\end{bmatrix} 
& D(k) &= \begin{bmatrix} 
k^{1+ \sigma/2} & 0 & 0 & 0 \\
0 & k^{\sigma/2} & 0 & 0 \\
0 & 0 & k^{-\sigma/2} & 0 \\
0 & 0 & 0 & k^{-1-\sigma/2} 
\end{bmatrix} 
\end{align*}
In the above, certain ``powers'' must be suitably interpreted. For example $k^{\sigma/2+1}$ means $k\sqrt{k^\sigma}$. Note that in the case $F = \F_q$ we have $\sigma(a) = a^t$ for $a \in \F_q$.


From Equations \eqref{x_std_gen} and \eqref{y_std_gen}, $\hat{x} = T(1, 0)$ and $\hat{y} = D(\omega)$. The following rules are straightforward to verify:
\begin{align}
T(a_1, b_1)T(a_2, b_2) &= T(a_1 + a_2, b_1 + b_2 + a_1 a_2^t) \\
T(a, b)^{D(\lambda)} &= T(\lambda a, \lambda^{t + 1} b). \label{eqn:T_conj}
\end{align}

When $(a, b, c, d) \neq (0, 0, 0, 0)$, we denote $\gen{(a, b, c, d)}$ by $(a : b : c : d)$. We will use the following well-known facts about the Suzuki groups, see \cite[Chapter $11$]{huppertIII} and \cite{baarnhielm05, MR0120283, suzuki62, MR0162840}.

\begin{lem} \label{lem_sz_facts}
Let $\Sz(q) = G \leqslant \GL(4, q)$.
\begin{enumerate}
\item $G = \gen{\hat{x}, \hat{y}, \hat{z}}$, and $\abs{G} = q^2 (q - 1) (q
  + t + 1)(q - t + 1)$, where the factors $q^2$, $q - 1$, $q
  + t + 1$, and $q - t + 1$ are relatively prime.
\item There exists $\OV \subset \F_q^4$ such that $G$ acts faithfully
  and doubly transitively on the projective points of
  \[ \OV = \set{(1 : 0 : 0 : 0)} \cup \set{(a b + a^{t + 2} + b^t : b : a : 1) \mid a,b \in \F_q} \]
  \emph{i.e.} the action is on elements of the set defined up to
  multiplication by a scalar. Denote $P_{\infty} = (1 : 0 : 0 : 0)$,
  $P(a, b) = (a b + a^{t + 2} + b^t : b : a : 1)$ and $P_0 = P(0, 0)$.
\item \label{sz_fact_point_stab}
  $\gen{\hat{x}, \hat{y}} \cong (\F_q^{+} {.} \F_q^{+}) {:} \F_q^{\times}$ is
  a maximal parabolic subgroup, and also the stabiliser of a point of $\OV$.
\item \label{sz_fact_point_stab_center} $\Cent_G(\hat{x}^2) = \gen{\hat{x}, \hat{y}}^{\prime} \cong \F_q^{+}
  {.} \F_q^{+}$, where $H^{\prime}$ denotes the commutator subgroup of a group $H$.
\item \label{sz_fact_even_order_center}
  $\Zent(\gen{\hat{x}, \hat{y}}^{\prime}) = \gen{\hat{x},
  \hat{y}}^{\prime\prime} \cong \F_q^{+}$, its non-identity elements
  are the involutions in $\gen{\hat{x}, \hat{y}}$, and $\gen{\hat{y}}$ acts transitively on $\Zent(\gen{\hat{x}, \hat{y}}^{\prime}) \setminus \gen{1}$.
\item All elements of $G$ of odd order, with the same matrix trace, are
  conjugate. There are $q + 3$ conjugacy classes of elements of odd order.
\item \label{sz_fact_even_order_conj} There are $3$ conjugacy classes of elements of even order, with representatives $T(0, 1)$, $T(1, 0)$ and $T(1, 1)$. These elements have trace $0$ and orders $2, 4, 4$.
\item \label{sz_fact_trace_5}

  All elements of $G$ of order $5$ have trace $1$, so they are all conjugate.
\item In $G$, the proportion of elements of order $q - 1$ is $\phi(q -
  1)/(2(q - 1))$, and in $\gen{\hat{x}, \hat{y}}$, the proportion is
  $\phi(q - 1)/(q - 1)$.
\item 
  \label{sz_fact_unique_rep}
Every element of $G$ has a unique representation as one of the forms $(T(a, b)\ D(l))$ or $(T(a, b) \ D(l) \ \hat{z}\ T(c, d))$ where $a,b,c,d,l \in \F_q$ and $l \neq 0$. If an element is represented as one of the two forms, then it cannot be represented as the other.
\item \label{sz_fact_elt_orders} Element orders in $\Sz(q)$ are divisors of some number in $\set{4, q - 1, q \pm t + 1}$.
\end{enumerate}
\end{lem}

The list of maximal subgroups of $\Sz(q)$ is also well-known, see \cite[Chapter $11$]{huppertIII}.

\begin{lem} \label{sz_maximal_subgroups}
A maximal subgroup of $G = \Sz(q)$ is conjugate to one of the following subgroups.
\begin{itemize}
\item The point stabiliser $\gen{\hat{x}, \hat{y}} = \Norm_G(\gen{\hat{x}, \hat{y}}^{\prime})$.
\item The normaliser $\Norm_G(\gen{\hat{y}}) = \gen{\hat{y}, \hat{z}} \cong \Dih_{2(q - 1)}$.
\item The normalisers $\mathcal{B}_i = \Norm_G(\gen{u_i})$ for $i = 1,2$, where $u_i \in G$ and $\abs{u_i} = q \pm t + 1$, satisfy $\mathcal{B}_i = \gen{u_i}{:}\gen{v_i}$ where $v_i \in G$, $\abs{v_i} = 4$ and $u_i^{v_i} = u_i^q$.
\item $\Sz(\hat{q})$ where $q$ is a prime power of $\hat{q}$ for $\hat{q} > 2$.
\end{itemize}
\end{lem}

From this we immediately obtain the following result.

\begin{cl} \label{sz_subgroups}
Proper subgroups of $\Sz(q)$ are either
\begin{itemize}
\item cyclic,
\item dihedral of order dividing $2(q - 1)$ or $2(q \pm t + 1)$,
\item $\Cent_n \cln \Cent_4$ where $n \mid q \pm t + 1$,
\item $\Sz(q^{\prime})$ where $q$ is a proper power of $q^{\prime} > 2$,
\end{itemize}
or contained in a point stabiliser.
\end{cl}

The following results are used when constructing $y$ in Theorem \ref{thm_sz_std_gens}.

\begin{lem} \label{lem_sz_hard_maximals} If $g, h \in G = \Sz(q)$
  satisfy $\abs{g} = 2$, $\abs{h} = 4$, $\abs{hg} = 4$ then
  $\abs{h^2 g} \in \set{1, 2}$ or $\abs{h^2 g} \mid (q \pm t + 1)$. If $\abs{h^2 g} = q \pm t + 1$ then $\gen{g, h} \cong \Cent_{q \pm t + 1} {:} \Cent_4$ is a
  maximal subgroup of $G$.
\end{lem}
\begin{proof}
  Clearly, $\gen{g, h}$ is an image in $G$ of the group $H = \gen{x, y
    \mid x^2, y^4, (xy)^4} \cong \Z^2 {:} \Cent_4$. Since $H$ is
  soluble, it is clear from Corollary \ref{sz_subgroups} that the
  specified orders are the only possible ones, and if $\abs{h^2 g} = q \pm t + 1$ then $\gen{g, h}$ must
  be one of the $\mathcal{B}_i$.
\end{proof}

\begin{lem} \label{lem_sz_maximal_trick}
Let $k$ be as in Lemma \ref{lem_primitive_elt}, and let $a = \omega^{k}$. There exists a unique $b \in \F_q$ such that $\abs{T(a, b) \hat{z}} = 4$ and $(T(a, b)^2 \hat{z})^{T(a, b)} = (T(a, b)^2 \hat{z})^q$.
\end{lem}
\begin{proof}
Calculations show that for every $x_1, x_2 \in \F_q$, we have $\abs{T(x_1, x_2) \hat{z}} > 2$ unless $x_1 = x_2 =
0$. Also $\Tr(T(x_1, x_2)^2 \hat{z}) = x_1^{t + 2}$ and \[\Tr(T(x_1, x_2) \hat{z}) = x_1^t + x_1^{t +
2} + x_1 x_2 + x_2^t.\]
If $\abs{T(a, x_2) \hat{z}} = 4$ for $x_2 \in \F_q$, then Lemma \ref{lem_sz_hard_maximals} implies that
$\gen{T(a, x_2), \hat{z}}$ is a subgroup of one of the
$\mathcal{B}_i$, since $T(a,x_2)^2 \hat{z}$ has non-zero trace and
hence has odd order. It remains to prove that there is a
unique $b \in \F_q^{\times}$ such that $T(a, b) \hat{z}$ acts as $v_i$
in Lemma \ref{sz_maximal_subgroups}.

In $\Sz(q)$, elements of trace $0$ have orders $1,2,4$, and
\begin{equation*}
\begin{split}
a^t + a^{t + 2} + ab + b^t &= 0 \  \Longleftrightarrow \\
a^2 + a^{2t + 2} + a^t b^t + b^2 &= 0 \ \Longrightarrow \\
b^2 + a^{t + 1} b + a^2 + a^{2t} &= 0,
\end{split}
\end{equation*}
where the third equation is $a^t$ times the first added to the second.
We obtain the solutions $b_1$ and $b_2$ from Lemma \ref{lem_quadratics}, and they both
give the value $a^{t+2}(1 + \Tr(a^{-1})) = 0$ of $\Tr(T(a, b)
\hat{z})$, since $\Tr(a^{-1}) = 1$. Hence $T(a, b_1) \hat{z}$ and
$T(a, b_2) \hat{z}$ have order $4$. One acts as the power $q$ on $T(a,
b)^2 \hat{z}$ and the other as the power $-q$.
\end{proof}

The following results are used when constructing $x$ in Theorem \ref{thm_sz_std_gens}.

\begin{lem} \label{lem_elements_order5}
The element $T(0, a) \hat{z}$ has order $5$ if and only if $a = 1$.
\end{lem}
\begin{proof}
  A calculation shows that $\abs{T(0, 1) \hat{z}} = 5$ and
  $\Tr(T(0, a) \hat{z}) = a^t$. Raising to the power $t$ is a field
  automorphism, hence $a \neq 1$ implies $\Tr(T(0, a) \hat{z}) \neq 1$
  and so $T(0, a) \hat{z}$ cannot have order $5$.
\end{proof}

\begin{lem} \label{lem_dihedral_trick}
The product of two random involutions in $\Sz(q)$ has odd order with probability at least $449/455$.
\end{lem}
\begin{proof}
  Given a fixed involution $j_1$ and a random involution $j_2$, by
  Lemma \ref{lem_sz_facts}, facts \ref{sz_fact_even_order_center} and
  \ref{sz_fact_even_order_conj}, we can assume $j_1 \in
  \Zent(\gen{\hat{x}, \hat{y}}^{\prime})$. If $j_2$ fixes the same
  point of $\OV$ as $j_1$, then by facts
  \ref{sz_fact_point_stab_center} and \ref{sz_fact_even_order_center},
  we have $j_2 \in \Zent(\gen{\hat{x}, \hat{y}}^{\prime})$, and so
  $\abs{j_1 j_2} = 2$ except in the case $j_2 = j_1$. If $j_2$ fixes a
  different point then by Corollary \ref{sz_subgroups}, $\gen{j_1,
    j_2}$ is a dihedral group of order $2k$ with $k$ odd, hence $j_1
  j_2$ has odd order. Hence the probability is $1 - 1/(q^2 + 1) +
  1/((q^2 + 1)(q-1)) \geqslant 449/455$ since $m \geqslant 1$.
\end{proof}

\begin{lem} \label{lem_main_step1}
If $g, h \in G = \Sz(q)$ satisfy that $\abs{g} = 4$, $\abs{h} = 2$, $g$ fixes $P \in \OV$ and $h$ fixes $Q \in \OV$ with $Q \neq P$, then there exists $u \in G$ such that $h^u = \hat{z}$ and $g^u = T(a, b)$ for some $a,b \in \F_q$.
\end{lem}
\begin{proof}
Note that $T(a, b)$ fixes $P_{\infty}$ and $\hat{z}$ fixes $P(1, 1) \neq P_{\infty}$. The result now follows immediately since $G$ is doubly transitive.
\end{proof}

The following results are used in the constructive membership testing.

\begin{lem} \label{lem_invol_formula} If $g \in \Sz(q)$ maps
  $P_{\infty}$ to $P(a_1, b_1) \notin \set{P_{\infty}, P_0}$, for some $a_1, b_1 \in \F_q$, then
  there exists unique $i \in \set{1, \dotsc, q - 1}$ such that
  $\abs{T(0, 1) T(0, 1)^{g D(\omega)^i}} = 5$. For this $i$, $T(0,
  1)^{g D(\omega)^i} = T(0, 1)^{\hat{z} T(a, b)}$ for some $(a, b)
  \neq (0, 0)$.
\end{lem}
\begin{proof}
  Since $g$ maps $P_{\infty}$ to $P(a_1, b_1)$, it follows that $g
  \notin \gen{\hat{x}, \hat{y}}$, and hence by Lemma \ref{lem_sz_facts}, fact \ref{sz_fact_unique_rep}, $g = T(a_2, b_2)
  D(\lambda_2) \hat{z} T(a_1, b_1) D(\lambda_1)$ for some $a_2, b_2 \in \F_q$. By Lemma \ref{lem_sz_facts}, fact \ref{sz_fact_trace_5}, the element $T(0, 1)
  T(0, 1)^{g D(\omega)^i}$ has order $5$ if and only if it has trace
  $1$, and a direct calculation shows that it has trace $1$ if and only
  if $\lambda_1 = \lambda_2 \omega^{-i}$. Therefore there is a unique such $i$.

For this $i$ and for some $\lambda \in \F_q^{\times}$
\begin{equation*}
\begin{split}
g D(\omega)^i &= T(a_2, b_2) D(\lambda) \hat{z} T(a_1, b_1) D(\lambda) = \\
&= T(a_2, b_2) \hat{z} T(a_1, b_1)^{D(\lambda)} = \\
&= T(a_2, b_2) \hat{z} T(a, b)
\end{split}
\end{equation*}
for some $(a, b) \in \F_q^2$, which is not $(0, 0)$ since $(a_1, b_1) \neq (0, 0)$.
\end{proof}


\begin{lem} \label{lem_slp_step2}
Let $(a, b) \in \F_q^2 \setminus (0, 0)$.
\begin{itemize}
\item There exists unique $c \in \F_q^{\times}$ such that
  \begin{equation} \label{eqn1_lem_slp_step2}
    \abs{T(0, c)^{\hat{z}} (\hat{x}^2)^{\hat{z} T(a, b)}} = 5.
  \end{equation}
\item For each such $c \in \F_q^{\times}$, if $a \neq 0$ then $d = a^{t + 1}$ is the unique element of $\F_q^{\times}$ such that 
\begin{equation} \label{eqn_lem_slp_step2}
\abs{T(0, c)^{\hat{z}} (\hat{x}^2)^{\hat{z} T(a, b) T(0, d)}} = 5.
\end{equation}
For every $a \in \F_q$, $d = 0$ also satisfies Equation \eqref{eqn_lem_slp_step2}.
\end{itemize}
\end{lem}
\begin{proof}
Recall that $\hat{x}^2 = T(0, 1)$. By Lemma \ref{lem_sz_facts}, fact \ref{sz_fact_trace_5}, an element has order $5$ if and only if its trace is $1$.
\begin{itemize}
\item The trace of $T(0, c)^{\hat{z}} (\hat{x}^2)^{\hat{z} T(a, b)}$ is $c^t (a^{t + 2} + ab + b^t)^2$, so the unique solution for $c$ in Equation \eqref{eqn1_lem_slp_step2} is $c = (a^{t + 2} + ab + b^t)^{-t}$.
\item In this case the trace of $T(0, c)^{\hat{z}} (\hat{x}^2)^{\hat{z} T(a, b) T(0, d)}$ is $1 + c^t (ad + d^t)^2$, so $ad = d^t$. Hence $d = 0$ is always a solution for Equation \eqref{eqn_lem_slp_step2}, and if $a = 0$ this is the only solution. If $a \neq 0$ then $a = d^{t - 1}$ is the solution.
\end{itemize}
\end{proof}

\section{Short presentation}

In \cite{MR2393425, MR2746771} a short presentation for $\Sz(q)$ is
given on four generators and 29 relations. Since we know of no
algorithm to construct these generators in an arbitrary representation
of $\Sz(q)$, its usefulness is limited. Here we give an explicit
presentation of $\Sz(q)$ on our standard generators. It has
$\TR{\log(q)}$ relations and can be verified in time $\OR{\log(q)}$
group operations. Some of the relations we use are critically used in
our constructive recognition algorithm. To develop this, we consider
the following general situation.

Let $F$ be a field of characteristic 2 equipped with an automorphism
$\sigma$ that satisfies $\lambda^{\sigma^2}=\lambda^2$ for all
$\lambda\in F$. We define the Suzuki group $\Sz(F)$ (or
$\mathrm{Sz}(F,\sigma)$ if $\sigma$ is not uniquely determined) to be
generated by the matrices $T(a,b)$ for $a,b\in F$, $D(k)$ for $k\in
F^\times$, and $z = \hat{z}$.

If $F$ is a subfield of the algebraic closure of $\F_2$ (including the
cases when $F$ is finite) then $\sigma$ is uniquely determined (if it
should exist at all).
Notation for group structures follows the
conventions in the {\sc Atlas} \cite{MR827219}, and we use $g^h$ and
$[g,h]$ to mean $h^{-1}gh$ and $g^{-1}h^{-1}gh$ respectively.

\subsection{The presentation}
\label{bray_suzuki_presentation}

We consider a group $G$ with generators $T(a,b)$ for $a,b\in F$, $D(k)$ for $k\in F^\times$, and $z$ subject to the relations
\begin{equation}\label{amalrels}
\begin{array}{r@{\;=\;}l@{\quad}l}
T(a,b)T(c,d)& T(a+c,b+d+ac^\sigma) & \mbox{for all }a,b,c,d\in F, \\
D(k)D(l) & D(kl) & \mbox{for all }k,l\in F^\times, \\
T(a,b)^{D(k)} & T(ak,bk^{1+\sigma}) & \mbox{for all }a,b\in F,k\in F^\times, \\
zD(k)z & D(k^{-1}) & \mbox{for all }k\in F^\times,
\end{array}
\end{equation}
and finally (as noted in \cite{MR2393425} this is relation is stated in \cite[p. 128]{suzuki62})
\begin{equation}\label{extrarel}
T(1,0)zT(0,1)zT(1,1)z = 1.
\end{equation}
Note that the relations in Equations \eqref{amalrels} and \eqref{extrarel} are satisfied by the matrices of the previous section, and observe also that the identities in Equation \eqref{amalrels}
imply that $T(0,0)=D(1)=1$. The first line of Equation \eqref{amalrels} implies that $U:=\{\,T(a,b):a,b\in F\,\}$ is a subgroup of shape
$(F,+).(F,+)$, while the second and third lines of Equation \eqref{amalrels} show that $H:=\{\,D(k):k\in F^\times\,\}$ is a subgroup
that normalises $U$, and moreover $H\cong (F^\times, \cdot)$. The second and final lines of Equation \eqref{amalrels} establish that $z^2=1$ and that $\langle H,z\rangle$ has dihedral type, with $z$ inverting all the elements of $H$. Thus
each element of $G$ can be written in the form
$dt_0zt_1z\ldots zt_n$ for some $n\in \N$, $d\in H$ and $t_0,\ldots,t_n\in U$.

We now show that we can take $n\leqslant 1$, and to do this it suffices to show that $zT(a,b)z$ has the form $dt_0zt_1$ or $dt_0$ for
all $a,b\in F$. Observations of \'Akos Seress (personal communication) were crucial to show that just one extra relation, namely Equation \eqref{extrarel}, was sufficient for this purpose. By considering various cyclic conjugates of Equation \eqref{extrarel} we obtain:
\begin{eqnarray}
zT(0,1)z = T(1,0)^{-1}z^{-1}T(1,1)^{-1} = T(1,1)zT(1,0), \label{EqT01} \\ 
zT(1,0)z = T(1,1)^{-1}z^{-1}T(0,1)^{-1} = T(1,0)zT(0,1), \label{EqT10} \\
zT(1,1)z = T(0,1)^{-1}z^{-1}T(1,0)^{-1} = T(0,1)zT(1,1). \label{EqT11} 
\end{eqnarray}
Conjugating the left hand side of Equation \eqref{EqT01} by $D(k^{-1})$ gives
\begin{equation}
D(k)zT(0,1)zD(k^{-1})=zD(k^{-1})T(0,1)D(k)z=zT(0,k^{1+\sigma})z, \label{EqT12}
\end{equation}

\noindent while the right hand side becomes
\begin{equation} \label{EqT14}
\begin{array}{r@{\;=\;}l}
D(k)T(1,1)zT(1,0)D(k^{-1}) & D(k)T(1,1)D(k)zT(k^{-1},0) \\
& D(k^2)T(k,k^{1+\sigma})zT(k^{-1},0). 
\end{array}
\end{equation}
Note that the \lq\lq power'' $1+\sigma$ is invertible, so if $b=k^{1+\sigma}$, then $k=b^{\sigma-1}$, and as $k$ runs through all of $F^{\times}$ then $b$ also runs through all of $F^{\times}$. From Equations \eqref{EqT12} and \eqref{EqT14} we then obtain
\begin{equation}
zT(0,b)z = D(b^{2\sigma-2})T(b^{\sigma-1},b)zT(b^{1-\sigma},0).\label{EqT0b}
\end{equation}
Conjugating the left hand side of Equation \eqref{EqT10} by $D(a^{-1})$ gives $zT(a,0)z$,
while the right hand side becomes
$$
\begin{array}{r@{\;=\;}l}
D(a)T(1,0)zT(0,1)D(a^{-1}) & D(a)T(1,0)D(a)zT(0,a^{-1-\sigma}) \\
& D(a^2)T(a,0)zT(0,a^{-1-\sigma}).
\end{array}
$$
Therefore for all $a\in F^\times$ we obtain:
\begin{equation}
zT(a,0)z = D(a^2)T(a,0)zT(0,a^{-1-\sigma}).\label{EqTa0}
\end{equation}
Combining Equations \eqref{EqTa0}, \eqref{EqT11} and \eqref{EqT0b} gives for all $a\ne 0,1$ that
\begin{equation*}
\begin{array}{r@{\;=\;}l}
zT(a+1,a+1)z & zT(a,0)z.zT(1,1)z \\
& D(a^2)T(a,0)zT(0,a^{-1-\sigma}).T(0,1)zT(1,1) \\
& D(a^2)T(a,0)zT(0,b)zT(1,1) \\
& D(a^2)T(a,0)D(b^{2\sigma-2})T(b^{\sigma-1},b)zT(b^{1-\sigma},0)T(1,1) \\
& D(a^2b^{2\sigma-2})T(ab^{2\sigma-2},0)T(b^{\sigma-1},b)zT(b^{1-\sigma},0)T(1,1),
\end{array}
\end{equation*}
where $b=a^{-1-\sigma}+1\ne 0$.
Since $zT(0,0)z=z^2=1$, it follows that $zT(a,b)z$ has the required form ($dt_0zt_1$ or $dt_0$) whenever $a=0$, $b=0$ or $a=b$.
If $a,b\ne 0$ then we set $k=(b/a)^{\sigma/2}$ and $c=a/k$, so that $(a,b)=(ck,ck^{1+\sigma})$.
Now $zT(c,c)z$ is of the form $dt_0zt_1$, so that \[zT(a,b)z=zT(ck,ck^{1+\sigma})z=D(k) zT(c,c)z D(k^{-1})\] is also of the
form $dt_0zt_1$.

We have thus shown how to reduce $zT(a,b)z$ to the form $d t_0$ or $dt_0zt_1$ for any $a,b\in F$. An explicit formula for this is somewhat complicated: for $(a,b)\ne (0,0)$ $$
zT(a,b)z = D(Q^{2-\sigma})T((a^{1+\sigma}+b)Q/Q^{\sigma},b)zT(b/Q,a/Q),
$$
where $Q=a^{2+\sigma}+ab+b^\sigma$, and $z T(0, 0) z = 1$.

\subsection{The finite Suzuki groups}
\label{bray_sz_presentation}

We now let $F=\F_q$, and let $\omega$ be a primitive element of $F$.
The automorphism $\sigma$ must be $\sigma:\lambda\mapsto \lambda^t$ for all $\lambda\in F$ and $t=\sqrt{2q}$ as in Section \ref{section:preliminaries}.

The presentation of the previous section translates into a `presentation' on $x$, $y$ and $z$ as follows.
$$
\langle\,x,y,z\mid x^4=y^{q-1}=z^2=(yz)^2=xzx^2zx^{-1}z=1, \langle x,y\rangle\cong \mathrm{E}_{q}^{1+1}\cn (q-1)\,\rangle.
$$
Clearly the specification that $\langle x,y\rangle\cong \mathrm{E}_{q}^{1+1}\cn (q-1)$ (the Borel subgroup) requires further
investigation. We shall let $x_i=y^{-i}xy^i$ for $i \geqslant 0$, which corresponds to the element $T(\omega^i,0)$ of $\Sz(q)$.

To try to ensure that $\langle x,y\rangle\cong \mathrm{E}_{q}^{1+1}\cn (q-1)$
we add relations. The first relation we shall add (other than $x^4=y^{q-1}=1$)
is the minimal polynomial relation
$$
x_{2m+1}^2x_{2m}^{2c_{2m}}\ldots x_2^{2c_2}x_1^{2c_1}x_0^2=1.
$$
where $X^{2m+1}+c_{2m}X+\cdots+c_2X^2+c_1X+1$ is the minimal polynomial of $\omega^{1+t}$ over $\F_2$ (note that necessarily
$c_{2m+1}=c_0=1$). This has the effect of forcing the subgroup generated by all of the $x_i^2$ (for $i\in \Z$) to be generated by the $x_i^2$ for $2m+1$ consecutive $i$.
Next we add the relations
$$
[x_0^2,x_i]=1\mbox{ for }1\leqslant i\leqslant 2m.
$$
Now $[x_0^2,x_0]=1$, and conjugating the above relations by suitable powers $y$ shows that $x_0$ commutes with $x_0^2$, $x_{-1}^2$, \ldots, $x_{-2m}^2$, and hence with every $x_i^2$. Conjugation by suitable powers of $y$ shows that $[x_i,x_j^2]=1$ for all $i,j\in \Z$. This ensures that all the $x_i^2$ are central, and that $\langle\,x_i^2:i\in \Z\,\rangle\cong \mathrm{E}_q$.

Let $X^{2m+1}+d_{2m}X+\cdots+d_2X^2+d_1X+1$ be the minimal polynomial of $\omega$ over $\F_2$ (where $d_i\in \{0,1\}$ for all $i$).

Let $\hat{p} = \hat{x}^{\hat{y}^{2m+1}} \hat{x}^{d_{2m}
  \hat{y}^{2m}} \dotsm \hat{x}^{d_1 \hat{y}} \hat{x}$ and
express $\hat{p}_{3, 1} = \sum_{i = 0}^{2m} e_i \omega^{i(t + 1)}$
where $e_i \in \F_2$. Note that $\omega^{i (t + 1)} =
((\hat{x}^2)^{\hat{y}^i})_{3, 1}$, so the basis $\set{\omega^{i(t
    + 1)} \mid 0 \leqslant i \leqslant 2m + 1}$ of $\F_q$ over $\F_2$
can be found from $\hat{x}$ and $\hat{y}$ only.

We add the relation
\begin{equation} \label{min_poly_order2}
x_{2m+1}^{\,}x_{2m}^{d_{2m}}\ldots x_2^{d_2}x_1^{d_1}x_0^{\,} = x_{2m}^{2e_{2m}}\ldots x_2^{2e_2}x_1^{2e_1}x_0^{2e_0}
\end{equation}

Finally, we add
$$
[x_0,x_i]=x_{i(t-1)}^2x_{i(2-t)}^2\mbox{ for }1\leqslant i\leqslant 2m.
$$


These relations enforce the structure of the Borel subgroup. To summarise, we have the following:
\begin{thm} \label{thm_short_pres}
\begin{equation} \label{sz_short_pres}
\Sz(q) \cong \gen{x, y, z \mid R \cup \set{z^2, (yz)^2, x z x^2 z x^{-1} z}},
\end{equation}
where $R = \set{x^4, y^{q - 1}} \cup R_1 \cup R_2$ and 
\begin{align*}
R_1 &= \bigcup_{i = 1}^{2m} \set{[x_0^2, x_i], [x_0, x_i] (x^2_{i (t - 1)} x^2_{i (2 - t)})^{-1}}, \\
R_2 &= \set{x^2_{2m+1} x_{2m}^{2c_{2m}} \dotsm x_1^{2c_1} x_0^2, x_{2m+1}^{\,}x_{2m}^{d_{2m}}\ldots x_2^{d_2}x_1^{d_1}x_0^{\,} = x_{2m}^{2e_{2m}}\ldots x_2^{2e_2}x_1^{2e_1}x_0^{2e_0}}. \\
\end{align*}

Given a generating set $Y$ of a black box group, with $\abs{Y} = 3$, in time $\OR{\mu \log q}$ it can be determined if $\gen{Y}$ satisfies the relations of Equation \eqref{sz_short_pres}.
\end{thm}
\begin{proof}
We can assume that the minimal polynomials of $\omega$ and $\omega^{1+t}$, and hence all
$c_i$ and $d_i$ are pre-computed. As noted above, the $e_i$ can also be
pre-computed from $\hat{x}$ and $\hat{y}$.

The relator $y^{q-1}$ requires time $\OR{\mu \log q}$ using exponentiation by squaring.

Each relator $i+1$ in $R_1$ can be computed from the relator $i$ using a
constant number of group operations. Hence all relations in $R$ can be
verified in time $\OR{\mu \log q}$.
\end{proof}


\section{Point stabiliser in $\Sz(q)$}

Recall from Lemma \ref{sz_maximal_subgroups} that a point stabiliser
in a Suzuki group corresponds to $\gen{\hat{x}, \hat{y}} < \Sz(q)$. We
now present an algorithm that constructs such a point stabiliser. This
is essentially \cite[Theorem 7]{MR3715222} extended to the black box
group context.

\begin{thm} \label{thm_bray_stab_trick} There exists a one-sided Monte
  Carlo algorithm with no false positives that, given a simple
  black box group $G = \gen{X}$ such that $G \cong \Sz(q)$, $\delta >
  0$, and an element $f \in G$ of order $4$ constructs $h \in G$ such
  that $\abs{h} = q - 1$ and $\gen{f, h}$ is a point stabiliser in
  $G$. The probability of failing to construct $h$ is less than $\delta$ and the
  algorithm has time complexity $\OR{N(\mu (\log q) +
    \xi + \eta)}$, where
\[N = \ceil{\frac{\log \delta}{\log (1 - \frac{\phi(q - 1)}{q - 1} (1 - 1/(q^2 + 1)) (1 - 3/q^2))}}. \]
If $f$ is given as
  an $\SLP$ in $X$ of length $\OR{n}$, then $h$ will be constructed as
  an $\SLP$ in $X$ of length $\OR{n + \log q}$.
\end{thm}
\begin{proof}
The algorithm proceeds as follows:
\begin{itemize}
\item Use the algorithm of \cite{bray00} to construct $g \in
  \Cent_G(f^2) \cong \F_q^{+} {.} \F_q^{+}$. If $\abs{g} = 4$ then $j :=
  g^2$, otherwise $j := g$. Repeat this step until $j \notin \set{1, f^2}$. The success probability is $(1 - 3/q^2)$ and this
  requires time $\OR{\mu + \xi}$.
\item Choose random $c \in G$ and check if $j^c \notin
  \Cent_G(f^2)$.
  As in the proof of Lemma \ref{lem_dihedral_trick}, the product of two involutions has odd
  order if they lie in distinct point stabilisers, so
  $\abs{f^2 j^c} = 2k + 1$ for some $k \in \N$. The number of
  involutions in $\Cent_G(f^2)$ is $q - 1$, and the number of
  involutions in $G$ is $(q - 1)(q^2 + 1)$, hence the probability of
  success is $1-1/(q^2 + 1)$ if $G \cong \Sz(q)$. Return to the first
  step otherwise. By Lemma \ref{lem_dihedral_trick}, this step
  requires time $\OR{\mu + \xi + \eta}$.
\item Let $h = c (f^2 j^c)^k$. This requires time $\OR{\mu
  \log q}$. Then $j^h = f^2$, so $h$ must fix the point fixed by $f^2$ and $j$,  and hence $\gen{f, h}$ lies in a point stabiliser.
\item Since $h \notin \Cent_G(f^2)$, it must have odd order, but $\gen{f, h}$ lies in a point stabiliser so $\abs{h} \mid q - 1$. Determine
  if $\abs{h} = q - 1$. If $\abs{h} \neq q - 1$, return to the
  first step. If $\abs{h} = q - 1$, then, by Lemma \ref{lem_sz_facts}, fact \ref{sz_fact_point_stab}, $\gen{f, h}$ must be the whole point stabiliser.
\end{itemize}
The probability that $h$ has the correct order is $\phi(q - 1) / (q -
1)$. Hence if we execute at most $N$ iterations, then the probability of
not finding $h$ is less than $\delta$, if $G \cong \Sz(q)$.

Note that we do not have to determine the order of $f^2 j^c$ to
compute $h$. It is sufficient to know an odd multiple of the order,
\emph{e.g.} $(q^2 + 1)(q - 1)$. This is because if $g$ has
order $(2k+1)$, but we only know a multiple $(2k + 1)(2l + 1) = 2(2kl
+ k + l) + 1$, then we can still compute $g^k$ since $g^{2kl + k + l}
= g^{l (2k + 1) + k} = g^k$.
\end{proof}

Later we shall need the following constructive membership testing in a point stabiliser.

\begin{thm} \label{thm_parabolic_slp} There exists an algorithm that,
  given standard generators $x, y \in G \cong \Sz(q)$ and $g \in U
  \geqslant G$, decides whether or not $g \in \gen{x, y}$, and if so,
  returns an $\SLP$ of $g$ in $\set{x, y}$ of length $\OR{\log q}$. The
  algorithm has time complexity $\TR{q \mu}$.
\end{thm}
\begin{proof}
  Clearly, if $g \in G$, then $g \in \gen{x, y}$ if and only if $(x^2)^g$
  centralises $x^2$, if and only if $[x^2, g]^2 = 1$. This provides a quick initial test. If $g \in G$, then $g$ corresponds to $T(a, b) D(w)$, for some $a, b \in \F_q$ and $w \in \F_q^{\times}$. To obtain an
  SLP for $g$, we reduce it to the identity, as follows:
\begin{itemize}
\item Find $k \in \set{1, \dotsc, q - 1}$ such that $(x^2)^g = (x^2)^{y^k}$. Return \texttt{false} if $k$ does not exist or if it is not unique. Let $g_1 = g y^{-k}$. Then $g_1$ corresponds to $T(a, b_1)$ for some $b_1 \in \F_q$ if $g_1 \in G$. This requires time $\TR{q \mu}$.

\item If $g_1^2 = 1$ then let $g_2 = g_1$ and skip to the next step. Otherwise find $k \in \set{1, \dotsc, q - 1}$ such that $g_1 x^{y^k} = 1$. Return \texttt{false} if $k$ does not exist or if it is not unique. Let $g_2 = g_1 x^{y^k}$. Then $g_2$ corresponds to $T(0, b_2)$ for some $b_2 \in \F_q$ if $g_2 \in G$. This requires time $\TR{q \mu}$.

\item If $g_2 = 1$ then let $g_3 = g_2$ and skip to the next step. Otherwise find $k \in \set{1, \dotsc, q - 1}$ such that $g_2 (x^2)^{y^k} = 1$. Return \texttt{false} if $k$ does not exist or if it is not unique. Let $g_3 = g_2 (x^2)^{y^k}$. Then $g_3 = 1$. This requires time $\TR{q \mu}$.

\item Now we have reduced $g$ to the identity, so we have an SLP $W$ for $g^{-1}$. Return \texttt{true} and $W^{-1}$.
\end{itemize}

If any of the steps fail, then $g \notin G$.
\end{proof}

\section{The main algorithm}
\label{section:main_alg}

We now describe the algorithm in Theorem \ref{thm_sz_std_gens}.
\begin{enumerate}
\item \label{main_alg_elt4_step} Choose at most
  \[N = \ceil{\frac{\log (\sqrt{\varepsilon})}{\log (1 - 1/q + 1/(q(q^2+1))))}}\]
  random pairs of elements $(f, c)$ from $G$, until $\abs{f} = 4$ and
  $\abs{s s^c} \notin \set{1, 2}$, where $s = f^2$. The total number
  of elements of order $4$ in $G$ is $q (q^2 + 1) (q - 1)$, and
  $\abs{G} = q^2 (q^2 + 1)(q - 1)$. As in the proof of Theorem
  \ref{thm_bray_stab_trick}, the probability of finding a suitable $c$
  is $1-1/(q^2 + 1)$. Hence, if $G$ is isomorphic to $\Sz(q)$, then the probability of not finding $f$ and $c$
  after $N$ trials, is less than
  $\sqrt{\varepsilon}$. In this case, return
  failure, since $G$ is probably not isomorphic to $\Sz(q)$. This requires time $\OR{N (\mu + \xi)}$.

\item \label{main_alg_z_step} Let $z = s^c$. Note that we have $f, s, z$ as SLPs of
  length $\OR{q}$. Moreover, by Lemma \ref{lem_main_step1} we can let
  $z$ and $f$ correspond to $\hat{z}$ and $T(a_1, b_1)$, so that $s$
  corresponds to $T(0, a_1^{t + 1})$, for some $a_1, b_1 \in \F_q$.

\item \label{main_alg_stab_step} Apply Theorem \ref{thm_bray_stab_trick} with $f$, and
  $\delta = \sqrt{\varepsilon}$, and hence construct $h$ of order $q -
  1$.
  Then $h$ corresponds to $T(a_2, b_2) D(\lambda)$, for some
  $a_2, b_2 \in \F_q$ and $\lambda \in \F_q^{\times}$.

\item \label{main_alg_f_standard_step} By Lemma \ref{lem_elements_order5} and Equation
  \eqref{eqn:T_conj} in Section \ref{section:preliminaries}, we can
  find $j \in \set{1, \dotsc, q - 1}$ such that $\abs{s^{h^j} z} =
  5$. Return failure if no such $j$ exists, since $G$ is not $\Sz(q)$.
  Replace $s$ and $f$ with $s^{h^j}$ and $f^{h^j}$. Then by Lemma
  \ref{lem_elements_order5}, $s$ corresponds to $T(0, 1)$ and by
  Equation \eqref{eqn:T_conj} $f$ corresponds to $T(1, b_3)$ for some
  $b_3 \in \F_q$. This requires time $\OR{q \mu}$.

\item We now want to replace $h$ so that its diagonal part corresponds to $D(\omega)$. For each $j \in \set{1, \dotsc, q - 1}$, perform the following steps:
\begin{itemize}
\item Determine if $\gcd(j, q - 1) = 1$, and skip to the next $j$
  otherwise. At the same time compute $1/j \pmod{q - 1}$. By
  \cite[Corollary $11.10$, Theorem $8.24$]{VonzurGathen03}, computing
  $\gcd(j, q - 1)$ and $1/j$ requires time \[\OR{\log q (\log\log q)^2
    \log\log\log q}.\]
\item Find the minimal polynomial $p(w) = \sum_{i = 0}^{2m+1} d_i w^i$
  of $\omega^j$, over $\F_2$, i.e. $d_i \in \F_2$ for $0 \leqslant d_i \leqslant 2m+1$. If $\lambda = \omega^j$, then from Equation \eqref{min_poly_order2} in Section \ref{bray_sz_presentation}, \[f^{d_{2m+1} h^{2m+1}} f^{d_{2m} h^{2m}} \dotsm f^{d_1 h} f^{d_0}\]
  has order at most $2$. If this is not the case, then skip to the next $j$, otherwise let $k := 1/j \pmod{q - 1}$ and break. By \cite{MR1802067}, this requires time $\OR{(\log q)^2 \zeta + \mu \log q}$.
\end{itemize}
This step requires time $\OR{q\log(q) (\zeta \log q + \mu + (\log\log q)^2 \log\log\log q)}$. If no such $k$ is found, then return failure, since $G$ is not $\Sz(q)$.
\item \label{main_alg_h_standard_step} Replace $h$ with $h^k$. Then $h$ corresponds to $T(a_2, b_2)
  D(\omega)$. This requires time $\OR{\mu \log q }$.
\item \label{main_alg_x_step} Find $j \in \set{1, \dotsc, q - 1}$ such that $\bar{f} := f
  s^{h^j}$ has order $4$ and such that either $z \bar{f} z \bar{f}^2
  z \bar{f}^3 = 1$ or $z \bar{f}^3 z \bar{f}^2
  z \bar{f} = 1$. From Section \ref{bray_suzuki_presentation}, this implies
  that $\bar{f}$ corresponds to $T(1, 0)$ or $T(1, 0)^{-1} = T(1, 1)$,
  respectively. If $z \bar{f} z \bar{f}^2 z
  \bar{f}^3 = 1$, then let $x = \bar{f}$, otherwise let $x
  = \bar{f}^{-1}$. This requires time $\OR{q \mu}$. Note that we then have $x$ as an SLP in $X$ of length $\OR{q}$, and $x$ corresponds to $\hat{x}$. If no such $\bar{f}$ is found, then return failure, since $G$ is not $\Sz(q)$.

\item \label{main_alg_omega_step} Compute $\Tr(\omega^i)$ for each $i = 0, \dotsc, 2m$. This requires time $\OR{(\log q)^2 \zeta}$.
\item Apply Lemma \ref{lem_primitive_elt} and construct a primitive element $a = \omega^l \in \F_q$ such that
  $\Tr(a^{-1}) = 1$. By \cite[Corollary $11.10$, Theorem
  $8.24$]{VonzurGathen03} and Equation \eqref{eqn_trace_comp} in Section \ref{section:preliminaries}, this requires time $\OR{q \log q}$. 

\item Let $u = x^{h^l}$. Then $u$ corresponds to $T(a, b_4)$ for
  some $b_4 \in \F_q$. This requires time $\OR{ \mu \log q}$.
\item For each $j \in \set{1, \dotsc, q - 1}$, let $\bar{u} := u s^{h^j}$ and find $j$ such that $\bar{u} z$ has order $4$ and $(\bar{u}^2 z)^{\bar{u}} = (\bar{u}^2 z)^q$. By Lemma \ref{lem_sz_maximal_trick}, this $j$ is unique. This requires time $\OR{q \mu}$. If no such $j$ is found, then return failure, since $G$ is not $\Sz(q)$.
\item Now $\bar{u}$ corresponds to $T(a, b)$ for some specific $b \in \F_q$, and by the proof of
  Lemma \ref{lem_sz_maximal_trick}, $b$ is either $\alpha = a^{t + 1}
  \sum_{i = 0}^{m + 1} a^{-2^i}$ or $\alpha + a^{t + 1}$. Determine
  which value is correct by performing the Lemma \ref{lem_sz_maximal_trick} calculations directly with the matrices in $\Sz(q)$. This
  requires time $\OR{ \zeta \log q}$.
\item \label{main_alg_disclog} Observe that $b^{t - 1} = \sum_{i = 0}^{2m} l_i \omega^i$ for some $l_i \in \F_2$. Let $v = \bar{u} (\prod_{i = 0}^{2m} x^{l_i h^i})^2$.
Then $v$ corresponds to $T(a, 0)$. This requires time $\OR{\mu \log q  + \zeta}$. 
\item From Equation \eqref{EqTa0} in Section \ref{bray_suzuki_presentation}, the following relation holds:
\[ z T(a, 0) z = D(a^2) T(a, 0) z T(0, a^{-1-t}) \]
Hence $\bar{v} := z v z s^{h^{-l}} z v^{-1}$ corresponds to $D(a^2)$.
\item \label{main_alg_y_step} Let $y = \bar{v}^{q / (2l)}$. Then $y$ corresponds to $\hat{y}$. Clearly we have $y$ as an SLP of length $\OR{q}$.
\end{enumerate}

The only steps that are probabilistic, and so may fail to construct required elements even if the input group is $\Sz(q)$, are steps \ref{main_alg_elt4_step} and \ref{main_alg_stab_step}. Hence the failure probability of the whole algorithm is less than $\varepsilon$. Finally, we can also verify
that $x, y, z$ are standard generators, using Theorem \ref{thm_short_pres}. This proves Theorem \ref{thm_sz_std_gens}. 

\section{Constructive membership testing}
\label{section:other_alg}

We now describe the algorithm of Theorem \ref{thm_sz_membership}. At
several steps we need to determine if an element has order $5$. No
order oracle is required to do this, since it is sufficient to
determine if the element is not the identity and that its fifth power
is the identity. We will denote the (assumed) image of $g$ in the standard copy by $\hat{g}$.

\begin{enumerate}
\item Use Theorem \ref{thm_parabolic_slp} to determine if $g$ or $gz$
  lies in $\gen{x, y}$. If so return \texttt{true} and the
  corresponding SLP for $g$.
\item Now $\hat{g}$ fixes $P(a_1, b_1) \notin \set{P_{\infty}, P_0}$, so $(a_1, b_1) \neq (0, 0)$. Find $j \in \set{1, \dotsc,
    q - 1}$ such that 
$\abs{x^2 (x^2)^{g y^j}} = 5$. By Lemma \ref{lem_invol_formula}, this $j$ is unique if $g \in G$, so return
  \texttt{false} if $j$ does not exist or if it is not unique. This requires time $\TR{q \mu}$. Replace $g$ with $g y^j$.
\item By Lemma \ref{lem_invol_formula}, $(\hat{x}^2)^{\hat{g}} = (\hat{x}^2)^{\hat{z} T(a, b)}$, and we want to find $a, b \in \F_q$. Find $i \in \set{1, \dotsc,
    q - 1}$ such that $\abs{(x^2)^{y^i z} (x^2)^g} = 5$. By Lemma \ref{lem_slp_step2}, this $i$ is unique if $g \in G$, so return
  \texttt{false} if $i$ does not exist or if it is not unique. This requires time $\TR{q \mu}$. Let $c = \omega^{i (1 + t)}$, so that $\abs{T(0, c)^z T(0, 1)^{\hat{g}}} = 5$.
\item Find $k \in \set{1, \dotsc,
    q - 1}$ such that $\abs{(x^2)^{y^i z} (x^2)^{g (x^2)^{y^k}}} = 5$. This takes time $\TR{q \mu}$. If $k$ cannot be found, then $a = 0$ by Lemma \ref{lem_slp_step2}. Return
  \texttt{false} if $k$ is not unique. Otherwise $a = \omega^k$.
\item From the proof of Lemma \ref{lem_slp_step2} we know that $c^t
  (a^{t + 2} + ab + b^t)^2 = 1$, and we know both $a$ and $c$. Now consider
\begin{equation*}
\begin{split}
a^{t + 2} + ab + b^t + c^{-t/2} &= 0 \Longleftrightarrow \\
a^{2t + 2} + a^t b^t + b^2 + c^{-1} &= 0 \Longrightarrow \\
b^2 + a^{t+1} b + c^{-1} + a^t c^{-t/2} &= 0,
\end{split}
\end{equation*}
where the third equation is $a^t$ times the first added to the second.
This is a quadratic equation in $b$, and we obtain the solutions $b_1$ and $b_2$ using Lemma \ref{lem_quadratics}. This requires time $\OR{\log(q) \zeta}$.

\item \label{other_alg_disclog} Observe that $b_1^{t - 1} = \sum_{i = 0}^{2m} l_i \omega^i$ for some $l_i \in \F_2$. Let $u_1 = x^{y^k} (\prod_{i = 0}^{2m} x^{l_i y^i})^2$ and let $u_2 = u_1^{-1}$. By Equation \eqref{eqn:T_conj} in Section \ref{section:preliminaries}, $\hat{u_1} = T(a, b_1)$ and $\hat{u_2} = T(a, b_2)$. This requires time $\OR{\mu \log q}$.
\item Let $u \in \set{u_1, u_2}$ be such that $(x^2)^g = (x^2)^{z u}$. Return \texttt{false} if none exists. Then $h = g u^{-1} z$ centralises $x^2$, so $h \in \gen{x, y}$. Use Theorem \ref{thm_parabolic_slp} (first step not necessary) to obtain an SLP of $h$, and hence an SLP of $g$.
\end{enumerate}

This completes the proof of Theorem \ref{thm_sz_membership}.

\begin{rem}
  Step \ref{main_alg_disclog} in Section \ref{section:main_alg} and
  step \ref{other_alg_disclog} in Section \ref{section:other_alg} can
  also be performed using a discrete logarithm oracle. In the first
  case, find $r \in \Z$ such that $b = \omega^r$. Then let $v =
  \bar{u} (x^2)^{h^r}$, which corresponds to $T(a, 0)$. This takes
  time $\OR{\mu \log q + \chi_D}$, and the advantage is that the
  SLP of $v$ is a factor $\OR{\log q}$ shorter.

  The other case is analogous.

\end{rem}


\section{Implementation and performance}
The algorithms in Theorem \ref{thm_sz_std_gens} and
\ref{thm_sz_membership} have been implemented in $\MAGMA$. To illustrate that the
algorithm for Theorem \ref{thm_sz_std_gens} is very practical, in Table \ref{tbl:matrix_groups_8}, \ref{tbl:perm_groups_8}, \ref{tbl:matrix_groups_16} and \ref{tbl:perm_groups_16} we display the timings when this
algorithm is executed on the representations of $\Sz(8)$ and $\Sz(32)$
that are available in version\ $3$ of the
\textsc{Web-Atlas} \cite{atlas_www}. The time shown is the average
taken over $100$ executions of Theorem \ref{thm_sz_std_gens} and
\ref{thm_sz_membership}, respectively. 

We have also benchmarked the algorithms for Theorem \ref{thm_sz_std_gens} and
\ref{thm_sz_membership} on the natural representations
for each field size. For each field size $q = 2^{2m + 1}$, the average
times over $100$ executions were recorded. The results are shown in
Figure \ref{fig:bb_natural_rep}. These graphs support the
stated time complexity.

All tests of Theorem \ref{thm_sz_std_gens} used $\varepsilon = 2^{-100}$.

The benchmark was performed on using $\MAGMA$ V2.28-21, Intel64-AVX2-Cuda11
flavour on an Intel i5-12600K 3.6GHz CPU and 16GB of RAM.


\begin{table}[hb!]
\captionof{table}{Timings for matrix groups for $q = 8$}
\label{tbl:matrix_groups_8}
\begin{tabular}{l|l|l|l}
Degree & Ring & Thm \ref{thm_sz_std_gens} [s] & Thm \ref{thm_sz_membership} [s] \\
\hline
64 & $\F_2$ & 0.001 & 0.001 \\
4 & $\F_{2^3}$ & 0.0003 & 0.0002 \\
16 & $\F_{2^3}$ & 0.0008 & 0.0008 \\
14 & $\F_5$ & 0.0008 & 0.0006 \\
35 & $\F_5$ & 0.0034 & 0.0045 \\
63 & $\F_5$ & 0.056 & 0.0074 \\
195 & $\F_5$ & 0.0275 & 0.0382 \\
65 & $\F_{5^3}$ & 0.099 & 0.1473 \\
64 & $\F_7$ & 0.0061 & 0.0083 \\
91 & $\F_7$ & 0.0119 & 0.016 \\
105 & $\F_7$ & 0.0105 & 0.0139 \\
14 & $\F_{7^2}$ & 0.0017 & 0.002 \\
14 & $\F_{13}$ & 0.0017 & 0.0022 \\
35 & $\F_{13}$ & 0.0053 & 0.008 \\
65 & $\F_{13}$ & 0.0173 & 0.026 \\
91 & $\F_{13}$ & 0.0339 & 0.0539 \\
14 & $\Z[i]$ & 0.046 & 0.0688 \\
64 & $\Z$ & 0.0076 & 0.0101 \\
65 & $\C$ & 0.1738 & 0.099 \\
91 & $\Z$ & 0.1959 & 0.25330 \\
105 & $\Z[i]$ & 15.288 & 22.652 \\
\end{tabular}

\vspace{1em}

\captionof{table}{Timings for permutation groups for $q = 8$}
\label{tbl:perm_groups_8}
\begin{tabular}{l|l|l}
Degree & Thm \ref{thm_sz_std_gens} [ms] & Thm \ref{thm_sz_membership} [ms] \\
\hline
65 & 0.2 & 0.1 \\
520 & 0.3 & 0.2 \\
560 & 0.4 & 0.2 \\
1456 & 0.6 & 0.6 \\
2080 & 0.8 & 0.8 
\end{tabular}

\vspace{1em}

\captionof{table}{Timings for matrix groups for $q = 32$}
\label{tbl:matrix_groups_16}
\begin{tabular}{l|l|l|l|l}
Degree & Ring & Thm \ref{thm_sz_std_gens} [s] & Thm \ref{thm_sz_membership} [s] \\
\hline
4 & $\F_{2^5}$ & 0.001 & 0.0013 \\
124 & $\F_5$ & 0.0364 & 0.0767 \\
124 & $\F_{41}$ & 0.175 & 0.4033 \\
124 & $\Z[i, 1/2]$ & 176 & 628
\end{tabular}

\vspace{1em}

\captionof{table}{Timings for permutation groups for $q = 32$}
\label{tbl:perm_groups_16}
\begin{tabular}{l|l|l}
Degree & Thm \ref{thm_sz_std_gens} [ms] & Thm \ref{thm_sz_membership} [ms] \\
\hline
1025 & 1.1 & 1.6 \\
198400 & 484.2 & 1007
\end{tabular}
\end{table}

\begin{figure}[ht!]
\caption{Benchmark of our implementations of the algorithms in Theorem \ref{thm_sz_std_gens} and \ref{thm_sz_membership} on the natural representation of $\Sz(q)$}
\includegraphics[scale=0.7]{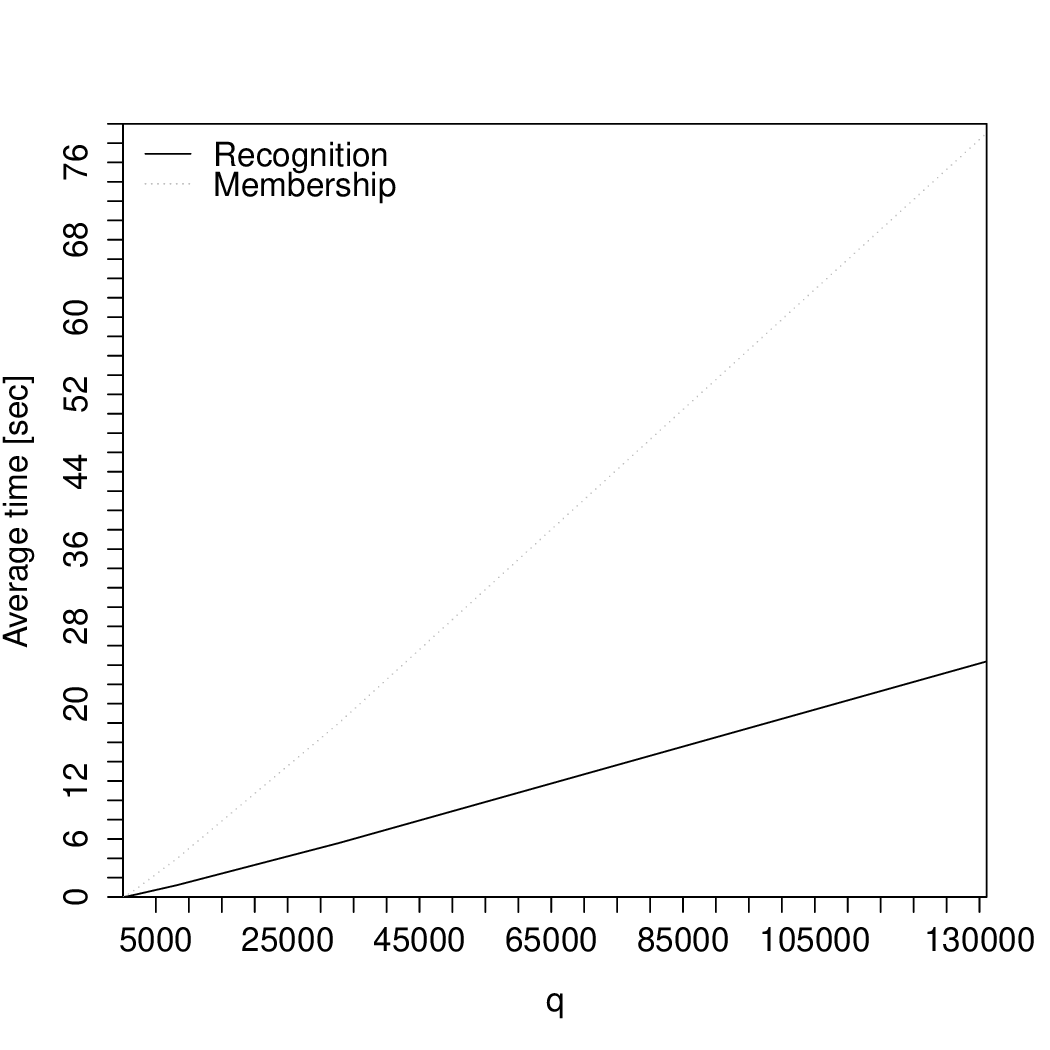}
\label{fig:bb_natural_rep}
\end{figure}

\FloatBarrier
\bibliographystyle{amsplain}
\bibliography{blackbox}

\end{document}